\documentclass[a4paper,11pt]{amsart}

\usepackage[top=3.3cm, bottom=3cm, left=2.5cm, right=2.5cm]{geometry}
\usepackage{latexsym,amsmath,amsthm,verbatim,ifthen,amssymb,mathtools,booktabs}

\usepackage{hyperref}
\usepackage{caption}
\usepackage{microtype}
\usepackage{graphicx, adjustbox, multirow, float}

\usepackage{enumerate}
\usepackage{palatino}
\usepackage{mathpazo}

\usepackage[all]{xy}
\usepackage{tikz}
\usepackage{tqft}
\usetikzlibrary{positioning,arrows,cd, patterns}

\newcommand{\id}{{\rm Id}}
\newcommand{\tc}{{\rm TC}}
\newcommand{\cd}{{\rm cd}}
\newcommand{\cat}{{\rm cat}}
\newcommand{\secat}{{\rm secat}}
\newcommand{\Agenus}{{\mathcal{A}\mbox{-}\rm genus}}
\newcommand{\Bgenus}{{\mathcal{B}\mbox{-}\rm genus}}
\newcommand{\genus}{{\rm genus}}

\newcommand{\NN}{\mathbb{N}}
\newcommand{\ZZ}{\mathbb{Z}}

\newcommand{\Aa}{{\mathcal{A}}}

\newcommand{\quot}[2]{\left.\raisebox{.2em}{$#1$}\middle/\raisebox{-.2em}{$#2$}\right.}

\newtheorem{theorem}{Theorem}[section]

\newtheorem{proposition}[theorem]{Proposition}

\newtheorem{corollary}[theorem]{Corollary}

\theoremstyle{definition}
\newtheorem{definition}[theorem]{Definition}

\theoremstyle{remark}
\newtheorem{example}[theorem]{Example}					
\newtheorem{remark}[theorem]{Remark}

\newtheorem{note}[theorem]{Note}

\numberwithin{equation}{section}

\begin{document}
	
	\title[$\secat(H \hookrightarrow G)$ and $\tc_r(\pi)$ as $\mathcal{A}$-genus]{Sectional category of subgroup inclusions and sequential topological complexities of aspherical spaces as $\mathcal{A}$-genus}
	
	\author[A. Espinosa Baro]{Arturo Espinosa Baro}
	\address{Faculty of Mathematics and Computer Science,
		Adam Mickiewicz University, Uniwersytetu Pozna\'nskiego 4, 61-614 Pozna\'n, Poland.} 
	\email{arturo.espinosabaro@gmail.com, artesp1@amu.edu.pl}
	
	\subjclass{55M30 (68T40, 55P91)}
	\keywords{A-genus, sectional category, sequential topological complexity, classifying space of a family of subgroups}
	
	\begin{abstract}

In this paper we characterize the sectional category of subgroup inclusions and the $r^{th}$-sequential topological complexity of aspherical spaces of a group $G$ in terms of the $\Agenus$ in the sense of Clapp-Puppe and Bartsch for a suitable one-element family of $G$-spaces $\Aa$, and we discuss some of the consequences of such characterization, including new ideas about notions of category-like invariants with respect to proper actions of groups.
	
	\end{abstract}
	
	\maketitle
	
	\setcounter{tocdepth}{1}
	
	\tableofcontents
	

		\section*{\centering Introduction}
		
The \textit{topological complexity} of a topological space was introduced by M. Farber in \cite{Farber03} in order to study the degree of instability of the motion planning of mechanical systems. The notion of \emph{sequential topological complexities} in turn was developed by Y. Rudyak in \cite{RudyakHigher} as a generalization of topological complexity which models the motion planning problem for robots that are supposed to make some pre-determined intermediate stops along their ways. These homotopy invariants, as well as the closely related and well-known Lusternik-Schnirelmann category, are particular instances of the more general \emph{sectional category} of fibrations, first defined by A. Schwarz in his seminal article \cite{Schwarz66}. 
		
There is a classic theorem of S. Eilenberg and T. Ganea stating that, for a torsion-free discrete group $\pi$, the Lusternik-Schnirelmann category of a $K(\pi,1)$ corresponds with the cohomological dimension of $\pi$, see \cite{EG65}. Motivated by such strong result, M. Farber pondered in \cite{Farber06b} whether or not a characterization of $\tc(K(\pi,1))$ is possible purely in terms of algebraic information coming from the group.  Such question is not by any means a trivial one. Even though the homotopy type of a $K(\pi,1)$-space (and hence all of its homotopy invariants) is completely determined by the group $\pi$, the description of such invariants may involve some homotopy-theoretical constructions that might not be expressible in terms of classifying spaces. As of today, the question remains open, and it has become one of the most relevant on the field of topological robotics. As such, the better understanding of the topological complexity of aspherical spaces (even when not on purely algebraic terms) is a prized objective.
				
In this paper, for $G$ a torsion-free discrete group, we give a characterization of the sectional category (and consequentially also of $\tc_r$) of fibrations between spaces of type $K(\pi,1)$ induced by subgroup inclusions in terms of the $\Agenus$ of their universal cover, for a suitable family of $G$-spaces $\Aa$. These characterizations have some interesting consequences and, while still topological in nature, they allow to refine some of the bounds derived from the characterization in terms of equivariant maps between classifying spaces coming from \cite{FGLO17}, \cite{FOSequ} and \cite{BCE22}, and to find in a rather direct manner new bounds for both sectional category and sequential topological complexities of aspherical spaces. 

We will also suggest, at the end of this paper, new ideas about notions of category-like invariants with respect to proper actions of groups, formulated in the language of $\Agenus$. 

\medskip
		
		\section*{Acknowledgements}
	The author was partially supported by  National Science Center, Poland research grant UMO-2022/45/N/ST1/02814 and by a doctoral scholarship of Adam Mickiewicz University. 
	The author wish to thank Zbigniew B{\l}aszczyk for motivation on this work, and also Antonio Viruel and Lucile Vandembroucq for many useful commentaries on a preliminary version of this manuscript.

\medskip		
\section{\centering Preliminaries}
		
\subsection{Sectional category and sequential topological complexities}

The \textit{sectional category} of a map $f \colon X \to Y$, written $\secat(f)$, is defined to be the smallest integer $n \geq 0$ such that there exists an open cover $U_0$, \ldots, $U_n$ of~$Y$ and continuous maps $s_i \colon U_i \to X$ with the property that $f \circ s_i$ is homotopic to the inclusion $U_i \hookrightarrow Y$ for any $0 \leq i \leq n$. 
		
\begin{proposition}[\cite{Schwarz66}] \label{SecatProperty1}
For a fibration $p \colon E \rightarrow B$, the following holds:
\begin{enumerate}[(a)]
\item Given another fibration  $p'\colon E' \rightarrow B$ such that there exists a commutative diagram $$\begin{tikzcd}
E \ar[r,"f"] \ar[d, swap, "p"] & E' \ar[d,"p'"] \\
B \ar[r, "="] & B
\end{tikzcd}$$ we have the inequality $$ 	\secat(p') \leq \secat(p). $$
		
\item $\secat(p) \leq k$ if and only if the $(k+1)$-fold fibrewise join of $p$ $$\underbrace{p * \cdots * p}_{k+1} \colon E *_B \cdots *_B E \to B$$ has a section, where the fiberwise join is defined as $$E *_B \cdots *_B E := \{ t_0x_0 + \cdots + t_nx_n \in E* \cdots *E \mid p(x_0) = \cdots = p(x_n) \}.$$

\end{enumerate}
\end{proposition}
		
In this work we are mostly interested in the sectional category of inclusion of subgroups $\iota \colon H \hookrightarrow G$, as introduced in \cite{BCE22}. Given such a group monomorphism, there is a covering map $ K(\iota,1) \colon K(H,1) \to K(G,1)$ between the corresponding Eilenberg--MacLane spaces induced by $\iota$ such that $\pi_1\big(K(\iota,1)\big)=\iota$.
				
\begin{definition}
The \emph{sectional category} of the subgroup inclusion $H \hookrightarrow G$, is defined as $$\secat(K(\iota,1) \colon K(H,1) \to K(G,1)).$$ 
\end{definition}
				
We briefly recall the definition of sequential topological complexities.
				
\begin{definition} \label{Def:TCr}
Let $X$ be a path-connected topological space. For each $r \in \NN$ with $r \geq 2$ the map $$p_r: PX\to X^r, \qquad p_r(\gamma) = \left(\gamma(0),\gamma\big(\tfrac{1}{r-1}\big),\gamma\big(\tfrac{2}{r-1}\big),\dots,\gamma\big(\tfrac{r-2}{r-1}\big), \gamma(1)  \right),$$
is a fibration. The \emph{$r$-th sequential topological complexity of $X$} is defined as
$$\tc_r(X) := \secat(p_r:PX \to X^r).$$
\end{definition}
				
Notice that, by definition, the notion of sequential topological complexity includes that of the classic topological complexity, which occurs as $\tc_2(X)=\tc(X).$ 

As noted in \cite{BGRT} the $r$-th sequential topological complexity of $X$ can be defined as 
\begin{equation}\label{EqTCrDiag}
\tc_r(X) = \secat(e^X_r)
\end{equation} the sectional category of the fibration $$\begin{tikzcd}[row sep=0pt,column sep=1pc]
e^X_r \colon X^{J_r} \arrow{r} & X^r \\
{\hphantom{(\gamma_0) \colon{}}} \gamma \arrow[mapsto]{r} & (\gamma(1_1), \cdots, \gamma(1_r)). 
\end{tikzcd} $$ where $J_r$ is the wedge of $r$ unit intervals $[0,1]$ (with $0$ as the base point for each of them), and $1_i$ stands for $1$ in the $i^{th}$ interval, for every $1 \leq i \leq r$. It is easy to see that indeed $e_r$ and $p_r$ are homotopy equivalent fibrations. Furthermore, as $e_r^X$ is the standard fibrational substitute for the iterated diagonal map $\Delta_{X,r} \colon X \rightarrow  X^r$ one can equivalently define $\tc_r(X) = \secat(\Delta_{X,r}).$
		
\subsection{Classifying spaces with respect to families of subgroups.}
		
A family $\mathcal{F}$ of subgroups of $G$ is said to be \textit{full} provided that it is non-empty and is closed under conjugation and subgroups. Given $H \leqslant G$, write $\langle H \rangle$ for the smallest full family of subgroups of $G$ containing $H$.
		
\begin{definition}
The \textit{classifying space of $G$ with respect to $\mathcal{F}$} is a $G$-CW complex $E_{\mathcal{F}}G$ satisfying the following conditions:
			
\begin{itemize}
\item every isotropy group of $E_{\mathcal{F}}G$ belongs to $\mathcal{F}$,
\item for any $G$-CW complex $X$ with all isotropy groups in $\mathcal{F}$ there exists a unique (up to $G$-equivariant homotopy) $G$-equivariant map $X \to E_{\mathcal{F}}G$. 
\end{itemize}
\end{definition} 
It is immediate, for any group $G$, that one can recover the usual total space of the universal bundle of $K(G,1)$ by putting $EG = E_{\{1\}}G$. From the definition, one observes that there is, up to $G$-homotopy, an unique canonically defined $G$-equivariant map $\rho \colon EG \to E_{\mathcal{F}}G$.
		
By the \textit{orbit category} of an abstract group $G$ associated to a family $\mathcal{F}$ of subgroups of~$G$, written $\textrm{Or}_{\mathcal{F}}G$, we understand a category whose object are homogeneous $G$-spaces $G/K$ for $K \in \mathcal{F}$, and morphisms are $G$-equivariant maps between them. An $\textrm{Or}_{\mathcal{F}}G$-\textit{module} is a contravariant functor from $\textrm{Or}_{\mathcal{F}}G$ to the category of abelian groups. An $\textrm{Or}_{\mathcal{F}}G$-\textit{homomorphism} of such modules is a natural transformation. In particular, the notion of a projective $\textrm{Or}_{\mathcal{F}}G$-module is defined.
		
The \textit{Bredon cohomological dimension} of $G$ with respect to $\mathcal{F}$, denoted $\cd_{\mathcal{F}}\,G$, is the length of the shortest possible $\textrm{Or}_{\mathcal{F}}G$-projective resolution of $\underline{\mathbb{Z}}$, where $\underline{\mathbb{Z}}$ is a constant $\textrm{Or}_{\mathcal{F}}G$-module which sends every morphism to $\textrm{id} \colon \mathbb{Z} \to \mathbb{Z}$. 
		
In \cite{FGLO17} the authors gave bounds of topological complexity of a space of type $K(\pi,1)$ in terms of the smallest dimension $n$ at which the canonical map between $E\pi$ and $E_{\langle \Delta_{\pi} \rangle}(\pi \times \pi)$ factors through the $n$-skeleton of $E_{\langle \Delta_{\pi} \rangle}(\pi \times \pi)$ and, as a consequence, in terms of Bredon cohomological dimension. In \cite{FOSequ} such bounds were generalized for sequential topological complexity with $r > 2$, and then to sectional category of arbitrary subgroup inclusions in \cite{BCE22}.
		
\subsection{$\Agenus$ of $G$-spaces}
		
A complete and thoroughful treatment of the fundamental properties of $\Agenus$ can be found in the classic book of Bartsch, \cite{Bartsch}.
		 
\begin{definition}
Let $G$ be a group and $\Aa$ a family of $G$-spaces. The $\Agenus$ of a $G$-space $X$ is defined as the smallest integer $k \geq 0$ such that there exists an open cover by $G$-invariant subsets $\{ U_0, \cdots, U_k \}$ of $X$ satisfying that, for every $0 \leq i \leq k$ there exists $A_i \in \Aa$ and a $G$-equivariant map $U_i \rightarrow A_i$.
\end{definition}
\begin{remark}
Naturally the above definition encompasses the non-equivariant setting by simply considering trivial $G$-actions. 
\end{remark}			
In particular, the notion of $\Agenus$ can be seen as a special case of the more general concept of $\Aa$-category, which in the equivariant context was due to M. Clapp and D. Puppe in \cite{ClP1}, as a way of generalizing the notion of Lusternik-Schnirelmann category and stuying critical point theory with respect to group actions, see \cite{ClP2}. For commodity of the reader, we proceed to recall the definition here:
\begin{definition}
Let $X$ and $Y$ be $G$-spaces. We define the $\mathcal{A}$-\textit{category} of a $G$-equivariant map $f: X \rightarrow Y$, denoted by $\mathcal{A}\mbox{-}\cat(f)$, as the smallest integer $k \geq 0$ such that there exists an open cover by $G$-invariant subsets $\{ U_0, \cdots, U_k \}$ of $X$ such that for every $0 \leq i \leq k$ there is some $A_i \in \mathcal{A}$ and $G$-equivariant maps $\alpha_i: X_i \rightarrow A_i$ and $\beta_i: A_i \rightarrow Y$ satisfying that the restriction $f_{|X_i}$ is $G$-homotopically equivalent to $\beta_i \circ \alpha_i$. If no such integer exists, we set $\mathcal{A}$-$\cat(f) = \infty.$

We define the $\mathcal{A}$-category \index{$\mathcal{A}$-category} of a $G$ space $X$ as $\mathcal{A}$-$\cat(X) := \mathcal{A}\mbox{-}\cat(\id_X)$.
\end{definition}

It is not difficult to see from the definition that the $\Agenus$ of a space can be recasted as the $\mathcal{A}$-category of the constant map, i.e. $$\Agenus(X) := \mathcal{A}\mbox{-}\cat(X \rightarrow \ast).$$

The relationship between different versions of topological complexity and the $\Aa$-category of the configuration space has been studied before. Indeed, N. Iwase and M. Sakai in \cite{Iwase10} first recasted $\tc(X)$ in terms of the $\Aa$-category of $X$ for a suitable family $\Aa$, while W. Lubawski and W. Marzantowicz in \cite{LubMar14} and A. Ángel, H. Colman, M. Grant and J. Oprea in \cite{ACGO20} studied the situation for equivariant versions of topological complexity. It is important to mention as well the work of P. Capovilla, C. Loeh and M. Moraschini, \cite{CLM}, where the relationship between a specialized case of the ideas of Clapp and Puppe (the amenable category, first introduced in \cite{Larranaga13}) and topological complexity is studied. However, to our knowledge, all these approaches have been only from the perspective of the $\Aa$-category of a space, and no attempt has been made to exploit the specificities of the properties of $\Agenus$ in order to study its applications to the setting of topological complexity.
	\vspace{0.2cm}
	
The following proposition summarizes some of the basic properties of $\Agenus$ that we will make use of later on (with a special emphasis on its characterization through the join). 
		
\begin{proposition}[\cite{Bartsch} Proposition 2.9, Proposition 2.15 and Proposition 2.17]\label{propgenus}
$\Agenus$ satisfies the following properties:
\begin{enumerate}[(a)]
\item $\Agenus(X)$ is the smallest integer $k \geq 0$ such that there exists $A_0, \cdots, A_k \in \Aa$ and a $G$-equivariant map $$X \rightarrow A_0 \ast \cdots \ast A_k.$$
\item \emph{Monotonicity:} If there exists a $G$-equivariant map $X \rightarrow Y$ then $\Agenus(X) \leq \Agenus(Y)$ 
\item \emph{Normalization:} If $A \in \Aa$ then $\Agenus(A) = 0$. 
\item Let $H \leq G$ closed, consider a family $\Aa$ of $G$-spaces and a family $\mathcal{B}$ of $H$-spaces. Then, for any $G$-space $X$ $$ \Bgenus(X) \leq (\Agenus(X) + 1)(\max \{ \Bgenus(A) \text{ : } A \in \Aa \} + 1) -1 $$
\end{enumerate}
\end{proposition}
\vspace{0.2cm}
Let $h^{\ast}$ be a multiplicative $G$-equivariant cohomology theory, and $I \subseteq h^{\ast}(\ast)$ an ideal.
\begin{definition}
The $(\Aa, h^{\ast},I)$-length of a $G$-space $X$ is defined as the smallest integer $k \geq 0$ such that there exist $A_0, \cdots, A_k \in \Aa$ satisfying that, for any $0 \leq i \leq k$ and $$\alpha_i \in I \cap \ker[h^{\ast}(\ast) \rightarrow h^{\ast}(A_i)]$$ we have $ p_{X}^{\ast}(\alpha_1) \cup \cdots \cup p_{X}^{*}(\alpha_k) = 0 $, where $p_X: X \rightarrow \ast$.
\end{definition}
We denote it by $l_{\Aa,h^{\ast},I}(X)$ (or simply $l(X)$ when the required inputs are clear). The length thus defined acts as a lower estimate for $\Agenus$ (see \cite[Corollary 4.9]{Bartsch}), i.e. $l_{\Aa,h^*,I}(X) \leq \Agenus(X)$.
  
 \section{Sectional category and $\tc_r$ as $\Aa$-genus}

A quick glance on both property $(b)$ of Proposition \ref{SecatProperty1} and the basic properties of $\Agenus$ described in the previous section allow us to immediately derive that $\Agenus$ constitutes a lower bound for general sectional category of fibrations.
 
\begin{proposition}
Let $p \colon E \rightarrow B$ be a fibration, and put $\Aa := \{ E \}$. Then we have $$ \Agenus(B) \leq \secat(p). $$
\end{proposition}
\begin{proof}
Let $k \in \NN_0$ be such that $\secat(p) \leq k$. By the aforementioned Proposition \ref{SecatProperty1} $(b)$, that means there exists a section of the $(k+1)$-fiberwise join of $p$, which means a map $$s \colon B \rightarrow \underbrace{E *_B \cdots *_B E}_{k+1} \hookrightarrow \underbrace{E* \cdots *E}_{k+1}$$ where the last arrow is just the natural inclusion of the fiberwise join into the iterated join. The claim then immediately follows from $(a)$ of Proposition \ref{propgenus}.
\end{proof}
In this section we obtain the aforementioned characterization of the sectional category of a group monomorphism and the $r^{th}$-sequential topological complexity of a $K(G,1)$-space for every $r \geq 2$ in terms of $\Agenus$. In fact, we will prove a more general statement, that the sectional category of every connected covering can be seen as an $\Agenus$ for a suitable family $\Aa$. From there, we will easily infer the rest of characterizations. All the groups considered throughout this section will be taken as discrete.

Throughout the rest of this work, whenever we have a ``nice enough'' topological space $X$ (usually a path connected $CW$-complex) we will denote its universal cover by $\widetilde{X}$, with universal covering map $\widetilde{p_X} \colon \widetilde{X} \rightarrow X$. Naturally, we regard $\widetilde{X}$ as a $\pi_1(X)$-space by deck transformations. 
 
 \begin{theorem} \label{SecatAsAgenus}
 	Let $X$ be a path connected $CW$-complex. If $q: \widehat{X} \rightarrow X$ is a connected covering, then $$\secat(q) = \Agenus(\widetilde{X})$$ where $\Aa = \left\{ \quot{\pi_1(X)}{\pi_1(\widehat{X})}  \right\}$.
 \end{theorem}
 
 \begin{proof}
 	First recall that, by means of the associated bundle construction, we can regard the connected covering $q \colon \widehat{X} \rightarrow X$ as a bundle \begin{equation}\label{inducedbundle}
 	q_0: \widetilde{X} \times_{\pi_1(X)} \left(\quot{\pi_1(X)}{\pi_1(\widehat{X})} \right) \rightarrow X
 	\end{equation} associated to the principal $\pi_1(X)$-bundle of the universal covering $\widetilde{p_X} \colon \widetilde{X} \rightarrow X$.
 	
 	Take $\{ U_i \}_{0 \leq i \leq n}$ an open covering of $X$ such that for each $0 \leq i \leq n$ there exists $s_i \colon U \rightarrow \widehat{X}$ a local section of the fibration $q$ over $U_i$. Notice that the restriction of $q_0$ to $U_i$ is again a covering. Moreover, if  $q$ has a local section over $U$, then there is a naturally induced local section of $q_0$. By \cite[Chapter 4, Theorem 8.1]{Husemoller66}, the sections of $$q_0 \colon \widetilde{p_X}^{-1}(U_i) \times_{\pi_1(X)} \left ( \quot{\pi_1(X)}{\pi_1(\widehat{X})} \right)$$ are in one-to-one correspondence with $\pi_1(X)$-equivariant maps $$\widetilde{p_X}^{-1}(U_i) \rightarrow \quot{\pi_1(X)}{\pi_1(\widehat{X})}.$$ Then, by putting $\Aa := \left\{ \quot{\pi_1(X)}{\pi_1(\widehat{X})}  \right\}$ we observe that the family of sets $\{\widetilde{p_X}^{-1}(U_i)\}_{0 \leq i \leq n}$ constitutes an open cover of $\widetilde{X}$ for $\Agenus$, from where we observe that $$\secat(q) = \secat(q_0) \geq \Agenus(\widetilde{X}).$$
 	
 	For the reverse inequality, notice that Proposition \ref{SecatProperty1} (b) and the induced bundle \refeq{inducedbundle} indicates that $\secat(q)$ corresponds with the minimal integer $k \geq 0$ such that the $(k+1)$-fold fiberwise join fibration $$ \ast_X^{k+1} \left[ q_0 \colon \widetilde{X} \times_{\pi_1(X)} \left( \quot{\pi_1(X}{\pi_1(\widehat{X})} \right) \rightarrow X \right] $$ has a section. The map $$ \widetilde{X} \times_{\pi_1(X)} \ast^{k+1} \left[ \left( \quot{\pi_1(X)}{\pi_1(\widehat{X})} \right) \right] \rightarrow \ast_X^{k+1} \left[  \widetilde{X} \times_{\pi_1(X)} \left( \quot{\pi_1(X)}{\pi_1(\widehat{X})} \right) \right]  $$ given by the identifications $$ \pi_1(X) \cdot \left(x,\displaystyle{\sum_{i=0}^{k+1}} t_i g_i \pi_1(\widehat{X}) \right) \mapsto \displaystyle{\sum_{i=0}^{k+1}} t_i \pi_1(X) \cdot \left(x, g_i \pi_1(\widehat{X})\right) $$ can be easily seen as a homeomorphism, which allows us to identify $\ast_X^{k+1}(q_0)$ with a fibration of the form $$ \widetilde{X} \times_{\pi_1(X)} \ast^{k+1} \left[ \left( \quot{\pi_1(X)}{\pi_1(\widehat{X})} \right) \right] \rightarrow X $$ and once again by \cite[Chapter 4, Theorem 8.1]{Husemoller66}, sections of this fibration are in one to one correspondence with $\pi_1(X)$-equivariant maps of the form $$ \widetilde{X} \rightarrow \ast^{k+1} \left[ \left( \quot{\pi_1(X)}{\pi_1(\widehat{X})} \right) \right]. $$ Then statement (a) of Proposition \ref{propgenus} shows that $\Agenus(\widetilde{X}) \geq \secat(q_0)$, hence we conclude the desired equality $$ \secat(q) = \secat(q_0) = \Agenus(\widetilde{X}). $$
 \end{proof}

 An immediate consequence of the previous theorem and our definition of $\secat(H \hookrightarrow G)$ is the characterization of the sectional category of any subgroup inclusion in terms of the $\Agenus$. We make it explicit in the following corollary.
 
 \begin{corollary}\label{SecatGrpAsAgenus}
 	Let $G$ be a torsion-free group, and $H \leqslant G$. Then we have $$ \secat(H \hookrightarrow G) = \Agenus(EG) $$ where $\Aa = \{ G/H \}.$
 \end{corollary}
 \begin{proof}
 	By definition $\secat(\iota \colon H \hookrightarrow G)$ coincides with the sectional category of the fibration $$ K(\iota,1) \colon K(H,1) \rightarrow K(G,1). $$ The covering map $EG/H \rightarrow EG/G$ associated to $H$ provides an explicit fibration whose sectional category corresponds with $\secat(H \hookrightarrow G)$. Therefore, the claim follows immediately from Theorem \ref{SecatAsAgenus}. 
 \end{proof}
 
 \begin{remark} Notice that an alternative proof of Corollary \ref{SecatGrpAsAgenus} is also implicit in the proof of \cite[Theorem 1.2]{BCE22}. Indeed, such proof indicates that $\secat(H \hookrightarrow G)$ corresponds with the smallest integer $k \geq 0$ such that there exists an equivariant map of the form $$ EG \rightarrow \ast^{k+1}(G/H). $$ Thus the result follows from Proposition \ref{propgenus} (a).
 \end{remark}
 
 Now we turn our attention to the corresponding characterization for sequential topological complexities. For the case of the diagonal subgroup inclusion $\Delta_{\pi,r} \hookrightarrow \pi^r$, we can reason as follows
 
 \begin{theorem}\label{TCrasAgenus}
 	Let $r \in \NN$ with $r \geq 2$, and let $X$ be a path connected $CW$-complex with $\pi_1(X) = \pi$. Put $$\Aa := \left\{ \quot{\pi^k}{\Delta_{\pi,r}} \right\}.$$ Then the following holds: 
 	\begin{enumerate}[(1)]
 		\item $\tc_r(X) \geq \Agenus(\widetilde{X}^r)$.
 		\item Furthermore, if $X$ is aspherical, then $\tc_r(X) = \Agenus(\widetilde{X}^r)$.
 	\end{enumerate}
 	
 \end{theorem}
 
 \begin{proof}
 	\begin{enumerate}[(a)]
 		\item Let $q \colon \widehat{X^r} \rightarrow X^r$ be the connected covering associated to the diagonal subgroup $\Delta_{\pi,r} \leqslant \pi^r$. Recall that $q_*(\pi_1(\widehat{X^r},\hat{x}_0)) \leqslant \pi_1(X^r,x_0)$ is either the diagonal group $\Delta_{X,r}$ or a subgroup of $\pi^r$ conjugated to it, depending just on the choice of basepoint $x_0 \in X$. Indeed, we can realize $\widehat{X^r}$ as $$\widehat{X^r} = \quot{\widetilde{X}^r}{\Delta_{\pi,r}},$$ i.e. as the orbit space of the $\Delta_{\pi,r}$-action on $\widetilde{X}^r$ obtained by restricting the $\pi^r$-action that is given as the $r$-fold product of the  $\pi$-action on $\widetilde{X}$ by deck transformations.
 		
 		As the fibration $e^X_r \colon X^{J_r} \rightarrow X^r$ is the standard fibrational substitute of the $r$-iterated diagonal map $\Delta_{X,r} \colon X \rightarrow X^r$, we know that, up to a choice of basepoint, $$(e^X_r)_*(\pi_1(X^{J_r})) = \Delta_{\pi,r}.$$ Therefore, we know that there exist basepoints $x \in X^{J_r}$ and $\widehat{x} \in \widehat{X^r}$ such that $$(e^X_r)_*(\pi_1(X^{J_r},x)) = q_*(\pi_1(\widehat{X^r},\widehat{x}))$$ and by the lifting criterion for coverings there exists a map $h \colon X^{J_r} \rightarrow \widehat{X^r}$ lifting the fibration $e^X_{r}$. As such, we can construct a commutative diagram of the form 
 		$$\begin{tikzcd}
 			X^{J_r} \ar[r,dashed,"h"] \ar[d, swap,"e^X_{r}"] & \widehat{X^r} \ar[d,"q"] \\
 			X^r \ar[r,"="] & X^r.
 		\end{tikzcd}$$ As a consequence of $(a)$ of Proposition \ref{SecatProperty1}, one observes that $$ \tc_r(X) \stackrel{\eqref{EqTCrDiag}}{=} \secat(e^X_r) \geq \secat(q) = \Agenus(\widetilde{X^r}) $$ where the last equality is a consequence of Theorem \ref{SecatAsAgenus} for the family $\Aa = \left\{ \quot{\pi^r}{\Delta_{\pi,r}} \right\}$. 
 		
 		\item Assume now that $X = K(\pi,1)$. Given the homotopy equivalence $X^{J_r} \simeq X$, we have that $X^{J_r}$ is an aspherical space. Under the hypothesis, we also know that the connected cover space $\widehat{X^r}$ is aspherical as well and, by the isomorphism induced at the level of fundamental groups, the map $h$ becomes an homotopy equivalence between $X \simeq X^{J_r}$ and $\widehat{X^r}$. Therefore we obtain the chain of equalities $$ \tc_r(X) = \secat(\Delta_r) = \secat(q \circ h) = \secat(q) = \Agenus(\widetilde{X^r}) $$ which finishes the proof. 
 	\end{enumerate} \end{proof}

 \begin{remark} \label{TCrcoverRemark} 
 	A couple of observations about the previous proof:
 	\begin{itemize}
 \item	If instead of using the fibrational substitute of the diagonal map we would want to proceed with the proof of Theorem \ref{TCrasAgenus} directly from the map $p_r \colon PX \rightarrow X^r$ as in Definition \ref{Def:TCr}, we can argue in the same way as noted in \cite[Section 6]{EFMO}. For convenience of the reader, we recall here the argument. One can define a continuous map of the form
 	$$\phi_X: PX \to \widehat{X^r}, \quad \phi(\gamma) = \rho(\tilde{\gamma}(0),\tilde{\gamma}(\tfrac{1}{r-1}),\dots,\tilde\gamma(\tfrac{r-2}{r-1}),\tilde\gamma(1)),$$
 	where here $\tilde\gamma$ denotes a lift of $\gamma$ to $\widetilde{X}$. This maps makes the following diagram commutative 
 	$$\begin{tikzcd}
 		PX \ar[dr, swap, "p_r"] \ar[rr, "\phi_X"] & & \widehat{X^r} \ar[dl, "q"] \\
 		& X^r &
 	\end{tikzcd}$$ Under the assumption of asphericity of $X$, the map $\phi_K$ becomes a fibre homotopy equivalence, and Dold's theorem informs us that $\phi_X$ is a fibre homotopy equivalence, so $\secat(p_r) = \secat(q)$, and we can conclude as above.
 	\item Note that the fact that $\tc_r(\pi)$ coincides with the sectional category of the covering of $(K(\pi,1))^r$  associated with the diagonal subgroup $\Delta_{\pi,r}$ was first proved in \cite[Lemma 4.2 and Corollary 4.3]{FOSequ}. Such proof, however, is performed under the hypothesis that $K(\pi,1)$ is finite, and through the identification of $\tc_r(\pi)$ with the sequential version of the $\mathcal{D}$-topological complexity. 
 	\end{itemize}
 \end{remark}
 
\medskip

 It is interesting as well to note that, as a consequence of the characterization, the usual zero-divisors cup length cohomological lower bound of (sequential) topological complexity can be recovered from the full generality of the properties of the length. Recall that the \emph{Borel equivariant cohomology} of a $G$-space $X$ is defined by $$H^*_G(X;M) = H^*(EG \times_G X;M). $$ Now take $h^*$ as the Borel equivariant cohomology theory, and $I = h^*(\ast) \cong H^*(B\pi^r)$. For any subgroup $H \leq \pi^r$ we have $h^*(\pi^r/H) \cong H^*(BH)$. Hence, for the iterated diagonal inclusion $\Delta_{\pi,r} \hookrightarrow \pi^r$ this becomes
$$ I \cap \ker\left[h^*(\ast) \rightarrow h^*\left(\quot{\pi^r}{\Delta_{\pi,r}}\right) \right] = \ker[H^*(B\pi^r) \rightarrow H^*(B \Delta_{\pi,r})], $$ and the induced map $p^*_{E \pi^r}$ is an isomorphism. 
 
 \section{Consequences of the characterization}
  
 The characterization of sectional category of connected coverings given in Theorem \ref{SecatAsAgenus} allows as well to derive new bounds for $\secat(H \hookrightarrow G)$ in terms of genus of classifying spaces for subgroup families and Bredon cohomological dimension, as we show in the following theorem.
 
 \begin{theorem} \label{Newbounds}
 	Let $G$ be a torsion-free group, $H \leqslant G$ and $\Aa = \{ G/H \}$.
 	\begin{enumerate}[(a)]
 		\item  For every full family of subgroups of $G$, namely $\mathcal{F}$, we have that $$ \secat(H \hookrightarrow G) \leq \Agenus(E_{\mathcal{F}}(G)). $$
 		\item For any subgroup $K \leqslant G$ subconjugate to $H$ such that $\cd_{\langle K \rangle} G \geq 3$ we have $$ \secat(H \hookrightarrow G) \leq \cd_{\langle K \rangle} G. $$
 		\item Under the hypothesis of (b), if $K \trianglelefteq G$ then $$ \secat(H \hookrightarrow G) \leq \cd(G/K). $$
 	\end{enumerate} 
 	
 \end{theorem}
 
 \begin{proof}
 	\begin{enumerate}[(a)]
 		\item 	By the universal property of the classifying space with respect to a full family of subgroups, for any such family $\mathcal{F}$ there exists a $G$-equivariant map (unique up to $G$-homotopy equivalence) $ \rho \colon E G \rightarrow E_{\mathcal{F}}(G). $ By virtue of monotonicity of $\Agenus$ (property (b) of Proposition \ref{propgenus}) and Theorem \ref{TCrasAgenus}, this yields $$ \secat(H \hookrightarrow G) = \Agenus(E G) \leq \Agenus(E_{\mathcal{F}}(G)). $$
 		
 		\item Let $K \leqslant G$ subconjugated to $H$ (i.e., such that there exists some $g \in G$ with $g^{-1}Kg \leqslant H$) and put $n := \cd_{\langle K \rangle}G$. As described by Blowers in \cite{Blowers}, the infinite join of the coset space $G/K$ is a model for the classifying space with respect to the family $\langle K \rangle$, i.e. $$ E_{\langle K \rangle}(G) \simeq \ast^{\infty}(G/K). $$ By the generalized Eilenberg-Ganea theorem for families \cite[Theorem 0.1]{LM00}, see also \cite[Theorem 5.2]{Luck}, there exists an $n$-dimensional model of the space $E_{\langle K \rangle}(G)$. Then, the equivariant Whitehead Theorem (see for example \cite[{Chapter 1, Theorem 3.2}]{May}) implies the existence of a $G$-equivariant map $E_{\langle K \rangle}G \rightarrow \ast^{n+1}(G/K)$. This, coupled with the existence of the comparison map $\rho \colon EG \rightarrow E_{\langle K \rangle}$, produces in turn a $G$-equivariant map $$ E G \xrightarrow{\rho} E_{\langle K \rangle}G \rightarrow \ast^{n+1}(G/K). $$ But as the subgroup $K$ is taken as subconjugated to $H$, the maps $G/K \rightarrow G/H$ induce $G$-equivariant maps at the level of joins. Thus, there exists a final $G$-equivariant map $$E G \rightarrow \ast^{n+1}(G/K) \rightarrow \ast^{n+1}(G/H). $$ The claim now follows from the characterization of $\Agenus$ in terms of the join, property (a) of Proposition \ref{propgenus}.
 		
 		\item Finally, under the hypothesis of statement (b), assume further that the subgroup is normal, $K \trianglelefteq G$. Then \cite[Proposition 4.21 and Corollary 4.22]{ArcinCisner17} informs us that $E(G/K)$ is in fact a model for $E_{\langle K \rangle}(G)$. Using (b) above we conclude the inequality $$ \secat(H \hookrightarrow G) \leq \cd_{\langle K \rangle}G = \cd(G/K).$$
 	\end{enumerate}
 	
 \end{proof}
 
 	For any $r \in \NN$ with $r \geq 2$ and by virtue of Theorem \ref{TCrasAgenus}, whenever we specialize Theorem \ref{Newbounds} to the case of the iterated diagonal inclusion $\Delta_{\pi,r} \hookrightarrow \pi^r$ we obtain, in an straightforward way the following new bounds for $\tc_r(\pi)$.
 	
 	\begin{corollary}\label{CorNewbounds}
 	Let $\pi$ be a torsion-free group, and $K \leqslant \pi^r$ subconjugated to $\Delta_{r,\pi}$. The following inequalities are satisfied:
 	\begin{enumerate}[(a)]
 		\item  $ \tc_r(\pi) \leq \Agenus(E_{\mathcal{F}}(\pi^r))$ for $\mathcal{F}$ any full family of subgroups of $\pi$.
 		\item $ \tc_r(\pi) \leq \cd_{\langle K \rangle} \pi^r. $
 		\item $ \tc_r(\pi) \leq \cd(\pi^r/K)$ if $K \trianglelefteq \pi^r$. 
 	\end{enumerate} 
 	\end{corollary} 
 	
 	\begin{remark}
 	Observe that, if in statement (b) of Corollary \ref{CorNewbounds} we take $H = \Delta_{r,\pi} \cong \pi$, we easily recover the upper bound from \cite[Corollary 3.2]{FOSequ}. Moreover, consider a central subgroup $H \leqslant \pi$. Its image through the iterated $r$-diagonal $\Delta_r(H) \leqslant \pi^r$ is a normal subgroup of $\pi^r$. Under these assumptions, statement (c) of Corollary \ref{CorNewbounds} recovers the well-known result from M. Grant, \cite[Proposition 3.7]{GrantFib} for $r = 2$, and gives a generalization for $r > 2$. 
 	\end{remark} \medskip
  
 An straightforward but interesting consequence of Theorem \ref{SecatAsAgenus}, Theorem \ref{TCrasAgenus} and property (d) of Proposition \ref{propgenus} is the following bound for sectional category and sequential topological complexities with respect to subgroups of a given group.

 \begin{corollary} \label{TCBgenus}
 	Let $\pi$ be a free torsion group with subgroups $H, K \leqslant \pi$, and $J \leqslant H$. The following inequality holds $$ \secat(J \hookrightarrow H) \leq (\secat(K \hookrightarrow \pi)+1)(\Bgenus(\pi/K)+1)-1 $$ for $\mathcal{B} = \{ H/J \}$.
 	
 	In particular, for the specific case of the iterated diagonal inclusions $\Delta_{H,r} \hookrightarrow H^r$ and $\Delta_{\pi,r} \hookrightarrow \pi^r$ the above inequality yields $$\tc_r(H) \leq (\tc_r(\pi) + 1)\left( \Bgenus \left(\quot{\pi^r}{\Delta_{\pi,r}} \right) + 1\right) -1  $$ where $\mathcal{B} = \left\{\quot{H^r}{\Delta_{H,r}} \right\}$ and $r \in \NN$ with $r \geq 2$. 
 \end{corollary}
 \begin{proof}
 	For the first part of the statement, consider $\Aa = {\pi/K}$ and $\mathcal{B} = \{ H/J \}$. By Theorem \ref{SecatAsAgenus} we have the identifications $$ \secat(K \hookrightarrow \pi) = \Agenus(E\pi), \qquad \secat(J \hookrightarrow H) = \Bgenus(E\pi)$$ which, combined with Proposition \ref{propgenus} (d), shows the claim.
 	
 	For the second part, if we define $$\Aa = \left\{ \quot{\pi^r}{\Delta_{\pi,r}} \right\} \qquad \mbox{ and } \qquad \mathcal{B} = \left\{ \quot{H^r}{\Delta_{H,r}} \right\},$$ then Theorem \ref{TCrasAgenus} (2) gives the characterizations $$ \tc_r(\pi) = \Agenus(E\pi), \qquad \tc_r(H) = \Bgenus(E\pi) $$ hence the result follows once again from applying Proposition \ref{propgenus} (d).
 \end{proof}
 
 Notice that the most significant aspect of this bound is that it gives us an explicit and computable measure of how much the sequential topological complexities of groups fail to be monotone under subgroup inclusions. In particular, one could exploit such measure to explore new examples of lower bounds with respect to subgroups. Here we present the case of the semidirect product of groups.

 \begin{corollary}\label{CorGenSemid}
 	Let $H$ and $K$ be torsion free groups. Then we have $\tc_r(H \rtimes K) \geq \tc_r(K)$ 
 \end{corollary}
 \begin{proof}
 	Put $G := H \rtimes K$. As the semidirect product $G$ is given by a group extension of the form $$ \{1\} \rightarrow H \rightarrow G = H \rtimes K \xrightarrow{p} K \rightarrow \{1\}  $$ and $K \leqslant G$, we have an induced $K$-equivariant map between the coset spaces $$ \quot{G^r}{\Delta_{G,r}} \xrightarrow{(p \times \cdots \times p)} \quot{K^r}{\Delta_{K,r}}. $$ Corollary \ref{TCBgenus} gives us the following inequality $$ \tc_r(K) \leq (\tc_r(G) + 1)\left( \Bgenus\left(\quot{G^k}{\Delta_{G,r}}\right) + 1\right) -1  $$ for the family $\mathcal{B} = \left\{ \quot{K^r}{\Delta_{K,r}}\right\}$. But statements (b) and (c) of Proposition \ref{propgenus} indicate that $$\Bgenus\left(\quot{G^r}{\Delta_{G,r}}\right) \leq \Bgenus\left({\quot{K^r}{\Delta_{K,r}}}\right) = 0 $$ so the claim follows.
 \end{proof}
 
 Compare our result with the alternative lower bound by the cohomological dimension given by M. Grant, G. Lupton and J. Oprea in \cite[Corollary 1.3]{GLO15} for $r = 2$. \bigskip
 
 For any subgroup $H \leqslant \pi \times \pi$ containing the diagonal subgroup $\Delta_{\pi}$ consider the family $\mathcal{B} = \left\{ \quot{(\pi \times \pi)}{H} \right\}$. Given that the natural projection $$\quot{(\pi \times \pi)}{\Delta_{\pi}} \rightarrow \quot{(\pi \times \pi)}{H}$$ is a $(\pi \times \pi)$-equivariant map, we see that $\Bgenus \left( \quot{(\pi \times \pi)}{\Delta_{\pi}}\right) = 0$. Then, by Proposition \ref{propgenus}(d) we get the inequality: $$ \Bgenus(E(\pi\times \pi)) \leq \Agenus(E(\pi \times \pi)) = \tc(\pi). $$ 
 
 Now, observe that there is a correspondence between subgroups of $\pi \times \pi$ containing the diagonal subgroup $\Delta_{\pi}$ and normal subgroups of $\pi$. To see this, first assume $\Delta_{\pi} \leqslant H \leqslant \pi \times \pi$, and define the subgroup of $\pi$ determined by $ K := \{ x \mid (1,x) \in H \}. $ This is clearly a normal subgroup, since $$(g,g)(1,x)(g^{-1},g^{-1}) = (1,gxg^{-1}) \qquad \forall (g,g) \in \pi \times \pi  $$ and the left hand side of the equality is in $H$, since $H$ is an overgroup of the diagonal subgroup $\Delta_{\pi}$. Now, if $K \trianglelefteq \pi$, we can easily define its associated overgroup of $\Delta_{\pi}$ by $$ H_K := \{ (g,h) \mid g^{-1}h \in K \}. $$ From here, notice that, since $$\Bgenus(E(\pi \times \pi)) = \secat(H \hookrightarrow \pi \times \pi),$$ we get a new series of bounds by Corollary \ref{SecatGrpAsAgenus} and repeated application of Proposition \ref{propgenus} (d) through a descending sequence of inequalities of the form $$ 0  \leq \cdots \leq \secat(H_{K_{i+1}} \hookrightarrow \pi \times \pi) \leq \secat(H_{K_{i}}\hookrightarrow \pi \times \pi) \leq \cdots \leq \tc(\pi \times \pi)  $$ where, for each index $i \geq 0$, the group $H_{K_{i}}$ is the overgroup of $\Delta_{\pi}$ associated to the normal subgroup $K_i \trianglelefteq \pi \times \pi$ for an ascending chain of normal subgroups of $\pi \times \pi$ of the form $$ \Delta_{\pi} = K_0 \leqslant K_1 \leqslant \cdots \leqslant K_i \leqslant K_{i+1} \leqslant \cdots \leqslant \pi \times \pi.$$ 
 
 \medskip

 \section{Thoughts on proper genus and proper topological complexity of groups}
 
 We will finish this paper with some very quick reflections on $\Agenus$ and topological complexity for proper actions of groups, that might be of interest for future work on the matter.  Let $G$ an arbitrary discrete group, this time possibly with torsion. Define the family of finite subgroups of $G$, denoted $$\mathcal{F}in := \{ H \leqslant G \mid |H| < \infty \}. $$ The classifying space with respect to such family is known as the \textit{classifying space for proper actions} \index{Classifying space!of a family!proper actions} of $G$, and it is usually denoted by $\underline{E}G$. Its orbit space with respect to the natural $G$-action  $\underline{E}G /G$ is denoted by $\underline{B}G$. This space receives the name of \textit{classifying space for proper $G$-bundles}, as it classifies proper $G$-bundles in an analogous way in which $BG$ classifies principal $G$-bundles (as shown in \cite{BCH94}). 
 
 Clearly, $\text{dim}_G(\underline{E}G) = 0$ if and only if $G$ is finite. In case $G$ is torsion free, it is also evident that $\underline{E}G \simeq EG$. So the truly interesting case for us is whenever $G$ is a non-finite group with torsion. As we know, for such groups the topological complexity is, indeed, infinite.  However, there is a powerful realization result in the spirit of Kan-Thurston involving the classifying space for proper actions, due to I. Leary and B. Nucinkis.
 
 \begin{proposition}[Theorem 1' of \cite{LeNu01}]\label{KTproper}
 	For any $CW$-complex $X$ there exists a group $G_X$ such that $\underline{B}G_X$ is homotopy equivalent to $X$. The group $G_X$ has a torsion-free subgroup $K_X$ of index two. If $X$ is finite, there is a finite model for $\underline{B}K_X$.
 \end{proposition} 
 
 Consequentially, every connected $CW$-complex has the homotopy type of a $\underline{B}G$. Moreover, as it is shown in the proof of Proposition \ref{KTproper}, the group is not torsion free except in the case that $X$ is one-dimensional. As such, the previous result provides us with a wide range of cases of torsion groups with associated spaces that may carry some non-infinite topological complexity. This result motivates our next definition.
 
 \begin{definition}
 	Define the $G$-\textit{proper topological complexity} \index{Topological Complexity!$G$-proper} of $G$, denoted by $\underline{\tc}(G)$, as $$\underline{\tc}(G) = \tc(\underline{B}G).$$
 \end{definition} 
 
In particular, if the group $G$ is locally finite, it is known that $\underline{B}G$ is contractible. As such, $\underline{\tc}(G) = 0$. Finding the complete characterization of groups for which $\underline{\tc}(G) = 0$ will be matter of future investigation. 

\begin{note}
It is important to distinguish this notion from the previously defined concept of \emph{proper topological complexity} (see \cite{GCalMur24}). That variant deals with different kind of information, and it is defined as an invariant of the proper homotopy type of the configuration space.
\end{note}

The main interest of $\underline{\tc}(G)$ lies in the fact that it allows us to recover the usual notion of topological complexity of the group $G$ whenever $G$ is torsion free, but it also potentially provides more information in good enough cases when $G$ has torsion but it has also good enough properties with respect to its proper actions. Namely, when $\underline{B}G$ can be realized by a finitely dimensional $CW$-complex model, by means of the procedure outlined in the proof of Proposition \ref{KTproper}).
 
 \begin{example}
 	Very often the problem lies within finding concrete examples. One of the possible approaches is through the usage of the nullification functor. Let $A$ and $X$ be topological spaces, we say that $X$ is $A$-null if the mapping space $\textrm{Map}(A,X)$ is homotopically equivalent to $X$ via $X \rightarrow \textrm{Map}(A,X)$, the inclusion of constant maps. We define the $A$-nullification of $X$ as an endofunctor $\mathbb{P}_A : \textrm{\textbf{Spaces}} \rightarrow \textrm{\textbf{Spaces}}$ that takes every space $X$ to a corresponding $A$-null space $\mathbb{P}_A(X)$ such that there exists a universal map $X \rightarrow \mathbb{P}_A(X)$ which induces a weak homotopy equivalence $\textrm{Map}(\mathbb{P}_A(X),Y) \simeq \textrm{Map}(X,Y)$ for every $A$-null space $Y$. It is possible to see that any other $A$-null space $U$ satisfying that property is homotopically equivalent to $\mathbb{P}_A(X)$. It is not within our objectives to develop further the machinery required or its properties, we refer the interested reader to \cite{DFarj96} for a thorough treatment of these tools. Denote now by $\mathcal{P} = \{2, 3, \cdots, p_i, \cdots \}$ the set of all prime numbers, and define the spaces $$W_n = \bigvee_{i=1}^{n} B \ZZ_{p_i} \qquad \mbox{and} \qquad  W_{\infty} = \bigvee_{p \in \mathcal{P}} B \ZZ_{p}.$$ The interesting point is that for any $G$ such that there exists a finite dimensional model for $\underline{B}G$, there is a homotopy equivalence $\underline{B}G \simeq \mathbb{P}_{W_{\infty}}(BG)$, see \cite[Theorem 3.2]{RF05}. Colourful examples through this technique can be found for some wallpaper groups (see \cite{Schatt78} for precise definitions), thanks to the computations carried out by R. Flores in \cite{RF05}, like:
 	
 	\begin{itemize}
	\item \textit{the discope group} $G = \mbox{pmm}$: $\underline{B}pmm \simeq \ast$ and thus $\underline{\tc}(pmm) = 0$.
 	\item  \textit{the tritrope group} $G = \mbox{p3}$: $\underline{B}\mbox{p3} \simeq S^2$, and thus $\underline{\tc}(\mbox{p3}) = 2$.
 	\end{itemize}
 	
 \end{example}
 \medskip
 
 \subsection{Proper genus of a group, and comparison with $\underline{\tc}(G)$.} 
 
 Given the family of subgroups $\mathcal{F}in$, denote its associated orbit category (and, by a slight abuse of notation, the collection of objects of such category as well) as $\mbox{Or}( \mathcal{F}in)$. Naturally, we can always consider the corresponding $\Agenus$ of an arbitrary discrete group $G$ with respect to the family of its finite subgroups, which motivates the following definition:

 \begin{definition}
 	Let $G$ an arbitrary group, and $\mathcal{F}in$ the closed family of finite subgroups. Define the \emph{proper genus} \index{$\Agenus$!proper genus} of the group $G$ by $$\underline{\genus}(G) := \mbox{Or}(\mathcal{F}in) \mbox{-}\genus(\underline{E}G).$$
 \end{definition}
 
 In view of the characterization from Theorem \ref{TCrasAgenus}, it is natural to ask whether the proper versions of topological complexity and genus coincide or not. Before adressing that matter, we will show a dimensional lower bound for the proper genus of a group. 
 
 \begin{proposition} \label{ProperGenusDim}
 	Let $G$ be a discrete group such that there is a finite dimensional model for $\underline{B}G$ satisfying $H^n(\underline{B}G;A) \neq 0$ for some $n \in \NN$ and some coefficient system $A$. Then we have $\underline{\mbox{genus}}(G) \geq n$. 
 \end{proposition}
 \begin{proof}
 	Let $k:= \underline{\genus}(G)$. By Proposition \ref{propgenus} (a) there exists a $G$-equivariant map $$ f \colon \underline{E}G \rightarrow \displaystyle{\ast^{k+1}_{i=0} (G/F_i)}$$ where for each $i \in \{ 0, 1, \cdots, k+1 \}$ the group $F_i$ belongs to the family $\mathcal{F}in$. The join space $\displaystyle{\ast^{k+1}_{i=0} (G/F_i)}$ is naturally a proper $G$-CW complex, where the isotropy subgroup of each point corresponds with the intersection group $$(t_0, a_0, t_1, a_1, \cdots, t_k, a_k) \in \displaystyle{\ast^{k+1}_{i=0} (G/F_i)} \qquad G_{(t_0, a_0, t_1, a_1, \cdots, t_k, a_k)} = \bigcap_{a_i} G_{a_i}. $$ By the universal property of $\underline{E}G$, there exists a $G$-equivariant map $$g \colon \displaystyle{\ast^{k+1}_{i=0} (G/F_i)} \rightarrow \underline{E}G$$ satisfying that $ g \circ f$ is  $G$-homotopically equivalent to the identity. Passing to the quotient, this map yields a homotopically commutative diagram of $CW$-complexes of the form $$\begin{tikzcd}
 		\underline{B}G \arrow[r, "\overline{f}"] \arrow[rd] & \quot{[\displaystyle{\ast_{i = 0}^{k+1}}(G/F_i)]}{G} \arrow[d, "\overline{g}"] \\
 		& \underline{B}G.
 	\end{tikzcd}$$ Due to the fact that each of the coset spaces $G/F_i$ is discrete for every $F_i \in \mathcal{F}in$, we observe that $$\dim \left( \quot{[\displaystyle{\ast_{i = 0}^{k+1}}(G/F_i)]}{G} \right) \leq \dim \left(\displaystyle{\ast_{i = 0}^{k+1}}(G/F_i)\right) = k $$ where $\dim$ just denotes the usual topological (or covering) dimension of the spaces. However, the induced composite map in cohomology $$ H^n(\underline{B}G;A) \xrightarrow{\overline{g}^*} H^n \left( \quot{[\displaystyle{\ast_{i = 0}^{k+1}}(G/F_i)]}{G};A \right) \xrightarrow{\overline{f}^*} H^n(\underline{B}G;A)  $$ is obviously an isomorphism, hence by dimensional reasons it must hold that $k \geq n$, which shows the claim.
 \end{proof}
 
We conclude with a couple of examples of this lower bound. 
 \begin{example}
 \begin{enumerate}[(a)]	
 \item Suppose a group $G$ such that its classifying space for proper $G$-bundles $\underline{B}G$ is a finite dimensional $CW$-complex with the homotopy type of a $n$-dimensional sphere $S^n$. By Proposition \ref{ProperGenusDim}, we have that $\underline{\genus}(G) \geq n$.  However, $$\underline{\tc}(G) = \tc(\underline{B}G) = \tc(S^n) = \begin{cases}
 1 & n \mbox{ odd}\\
 2 & n \mbox{ even}
 \end{cases} $$ Through this example, we can conclude that the difference between $\underline{\genus}(G)$ and $\underline{\tc}(G)$ may be arbitrarily large. 
 \item More generally, we can consider the case of virtual Poincar\'e duality groups. Recall that given a group $G$ we say that $G$ virtually satisfies a property if there exists a subgroup $H \leqslant G$ of finite index with such property. R. Kim proved a specialization of Proposition \ref{KTproper}, see \cite[Theorem 1, Corollary 2]{Kim13}, stating that for any finite connected simplicial complex $X$, there exists a virtually torsion-free group $G$ with $\underline{E}G$ a cocompact manifold such that $\underline{B}G$ is homotopy equivalent to $X$. Furthermore, he also shown that for any finite connected simplicial complex $X$, there is a virtual Poincar\'e duality group $G$ such that $\underline{B}G$ is homotopy equivalent to $X$. 
 				
 Thus, for a finite simplicial complex $X$ with $H^n(X;A) \neq 0$ for some $n > 1$ and some coefficient system $A$, let $H$ be a torsion free Poincar\'e duality subgroup of a group $G$, such that $|G:H| \leq \infty$, satisfying that $\underline{E}G$ is a cocompact manifold and $\underline{B}G \simeq X$, as in \cite[Theorem 1, Corollary 2]{Kim13}. By the cocompacity of $\underline{E}G$ we know that $\underline{B}G$ has a finite dimensional model, thus by Proposition \ref{ProperGenusDim}, we have $\underline{\genus}(G) \geq n$. 
 \end{enumerate}
 \end{example}

\bibliography{SecatasAgenus.bib}{}

\begin{thebibliography}{10}

\bibitem{ACGO20}
A.~Angel, H.~Colman, M.~Grant, and J.~Oprea.
\newblock Morita invariance of equivariant {L}usternik-{S}chnirelmann category
  and invariant topological complexity.
\newblock {\em Theory Appl. Categ.}, 35:179—195, 2020.

\bibitem{ArcinCisner17}
J.A. Arciniega-Nev\'{a}rez and J.L. Cisneros-Molina.
\newblock Comparison of relative group (co)homologies.
\newblock {\em Bol. Soc. Mat. Mex. (3)}, 23(1):41--74, 2017.

\bibitem{Bartsch}
T.~Bartsch.
\newblock {\em Topological methods for variational problems with symmetries}.
\newblock Lecture Notes in Mathematics, No. 1560. Springer-Verlag, Berlin,
  1993.

\bibitem{BGRT}
I.~Basabe, J.~Gonz\'{a}lez, Y.~B. Rudyak, and D.~Tamaki.
\newblock Higher topological complexity and its symmetrization.
\newblock {\em Algebr. Geom. Topol.}, 14(4):2103--2124, 2014.

\bibitem{BCH94}
P.~Baum, A.~Connes, and N.~Higson.
\newblock Classifying space for proper actions and {$K$}-theory of group
  {$C^*$}-algebras.
\newblock volume 167 of {\em Contemp. Math.}, pages 240--291. Amer. Math. Soc.,
  Providence, RI, 1994.

\bibitem{BCE22}
Z.~B{\l}aszczyk, J.~Carrasquel-Vera, and A.~{Espinosa Baro}.
\newblock On the sectional category of subgroup inclusions and {$A$}damson
  cohomology theory.
\newblock {\em Journal of Pure and Applied Algebra}, 226(6):106959, 2022.

\bibitem{Blowers}
J.V. Blowers.
\newblock The classifying space of a permutation representation.
\newblock {\em Trans. Amer. Math. Soc.}, 227:345--355, 1977.

\bibitem{CLM}
P.~Capovilla, C.~Loeh, and M.~Moraschini.
\newblock Amenable category and complexity.
\newblock {\em Algebr. Geom. Topol.}, 22:1417–1459, 2022.

\bibitem{ClP1}
M.~Clapp and D.~Puppe.
\newblock Invariants of {L}usternik-{S}chnirelmann type and the topology of
  critical sets.
\newblock {\em Trans. Amer. Math. Soc.}, 298:603--620, 1986.

\bibitem{ClP2}
M.~Clapp and D.~Puppe.
\newblock Critical point theory with symmetries.
\newblock {\em J. reine angew. Math.}, 418:1--29, 1991.

\bibitem{DFarj96}
E.~Dror~Farjoun.
\newblock {\em Cellular spaces, null spaces and homotopy localization}.
\newblock Lecture Notes in Mathematics 1622. Springer-Verlag, Berlin, 1996.

\bibitem{EG65}
S.~Eilenberg and T.~Ganea.
\newblock On the {L}usternik--{S}chnirelmann category of abstract groups.
\newblock {\em Ann. of Math.}, 65:517--518, 1957.

\bibitem{EFMO}
A.~{Espinosa Baro}, M.~Farber, S.~Mescher, and J.~Oprea.
\newblock Sequential topological complexity of aspherical spaces and sectional
  categories of subgroup inclusions.
\newblock To appear in Math. Ann., 2023.

\bibitem{Farber03}
M.~Farber.
\newblock Topological complexity of motion planning.
\newblock {\em Discrete Comput. Geom.}, 29(2):211--221, 2003.

\bibitem{Farber06b}
M.~Farber.
\newblock Topology of robot motion planning.
\newblock In {\em in: Morse Theoretic Methods in Nonlinear Analysis and in
  Symplectic Topology}, NATO Sci. Ser. II Math. Phys. Chem. 217, pages
  185--230. Springer, 2006.

\bibitem{FGLO17}
M.~Farber, M.~Grant, G.~Lupton, and J.~Oprea.
\newblock Bredon cohomology and robot motion planning.
\newblock {\em Algebr. Geom. Topol.}, 19, Number 4:2023--2059, 2019.

\bibitem{FOSequ}
M.~Farber and J.~Oprea.
\newblock Higher topological complexity of aspherical spaces.
\newblock {\em Topology Appl.}, 258:142--160, 2019.

\bibitem{RF05}
R.~Flores.
\newblock Nullification functors and the homotopy type of the classifying space
  for proper bundles.
\newblock {\em Alg. Geom. Topol.}, 46:1141--1172, 2005.

\bibitem{GCalMur24}
J.M. Garc\'ia-Calcines and A.~Murillo.
\newblock Proper topological complexity.
\newblock To appear in Topol. Methods Nonlinear Anal., 2024.

\bibitem{GrantFib}
M.~Grant.
\newblock Topological complexity, fibrations and symmetry.
\newblock {\em Topology Appl.}, 159:88--97, 2012.

\bibitem{GLO15}
M.~Grant, G.~Lupton, and J.~Oprea.
\newblock New lower bounds for the topological complexity of aspherical spaces.
\newblock {\em Topology and its Applications}, 189:78--91, 2015.

\bibitem{Larranaga13}
J.~Gómez-Larrañaga, Francisco González-Acuña, and H.~Wolfgang.
\newblock Amenable category of three-manifolds.
\newblock {\em Algebr. Geom. Topol.}, 13:905--925, 03 2013.

\bibitem{Husemoller66}
D.~Husemoller.
\newblock {\em Fibre {B}undles}.
\newblock Graduate Texts in Mathematics 20. Springer-Verlag, 1994.

\bibitem{Iwase10}
N.~Iwase and M.~Sakai.
\newblock Topological complexity is a fibrewise {L}-{S} category.
\newblock {\em Topology Appl.}, 157(1):10--21, 2010.

\bibitem{Kim13}
R.~Kim.
\newblock Every finite complex is the classifying space for proper bundles of a
  virtual {P}oincar\'e duality group.
\newblock {\em Mathematische Zeitschrift}, 275:673--679, 2013.

\bibitem{LeNu01}
I.~Leary and B.~Nucinkis.
\newblock Every {CW}-complex is a classifying space for proper bundles.
\newblock {\em Topology}, 40(3):539--550, 2001.

\bibitem{LubMar14}
W.~Lubawski and W.~Marzantowicz.
\newblock {Invariant topological complexity}.
\newblock {\em Bull. London Math. Soc.}, 47:101--117, November 2014.

\bibitem{Luck}
W.~Lück.
\newblock Survey on classifying spaces for families of subgroups.
\newblock In {\em Infinite Groups: Geometric, Combinatorial and Dynamical
  Aspects}, Progr. Math. 248, pages 269--322. Birkh{\"a}user, 2005.

\bibitem{LM00}
W.~Lück and D.~Meintrup.
\newblock On the universal space for group actions with compact isotropy.
\newblock In {\em Geometry and topology: {A}arhus (1998)}, volume 258 of {\em
  Contemp. Math.}, pages 293--305. Amer. Math. Soc., Providence, RI, 2000.

\bibitem{May}
J.P. May.
\newblock {\em Equivariant {H}omotopy and {C}ohomology {T}heory}.
\newblock CBMS Regional Conference Series in Mathematics 91. American
  Mathematical Society, 1996.

\bibitem{RudyakHigher}
Y.~B. Rudyak.
\newblock On higher analogs of topological complexity.
\newblock {\em Topology Appl.}, 157(5):916--920, 2010.

\bibitem{Schatt78}
D.~Schattschneider.
\newblock The plane symmetry groups: their recognition and notation.
\newblock {\em Amer. Math. Monthly}, 85:439--450, 1978.

\bibitem{Schwarz66}
A.~Schwarz.
\newblock The genus of a fiber space.
\newblock {\em Amer. Math. Soc. Transl.}, 55:49--140, 1966.

\end{thebibliography}
\bibliographystyle{plain}

\end{document}